\documentclass[twoside,12pt]{article}
\usepackage{amsmath, amsthm, amscd, amsfonts, amssymb, graphicx, color}
\usepackage{graphicx}
\begin{document}
\setcounter{page}{1}
\setlength{\unitlength}{12mm}
\newcommand{\f}{\frac}
\newtheorem{theorem}{Theorem}[section]
\newtheorem{lemma}[theorem]{Lemma}
\newtheorem{proposition}[theorem]{Proposition}
\newtheorem{corollary}[theorem]{Corollary}
\theoremstyle{definition}
\newtheorem{definition}[theorem]{Definition}
\newtheorem{example}[theorem]{Example}
\newtheorem{solution}[theorem]{Solution}
\newtheorem{notation}[theorem]{Notation}
\theoremstyle{remark}
\newtheorem{remark}[theorem]{Remark}
\numberwithin{equation}{section}
\newcommand{\sta}{\stackrel}
\title{\bf P\'{o}lya-Knopp and Hardy-Knopp type inequalities for Sugeno integral}
\author { B. Daraby$^{a}$, F. Rostampour$^{b}$, A. R. Khodadadi$^{c}$,  A. Rahimi$^{d}$, R. Mesiar$^{e}$}
\date{\footnotesize $^{a, b, c, d}$Department of Mathematics, University of Maragheh, P. O. Box 55136-553, Maragheh, Iran\\
$^{e}$Department of Mathematics and Descriptive Geometry, Slovak University of Technology, SK-81005 Bratislava, Slovakia,
Czech Academy of Sciences, Institute of Information Theory and Automation.
}
\maketitle

\begin{abstract}
\noindent
In this paper, we investigate the P\'{o}lya-Knopp type inequality for Sugeno integrals in two cases. In the first case, we suppose that the inner integral is the standard Riemann integral and the remaining two integrals are of Sugeno type. In the second case, all involved integrals are Sugeno integral. We present several examples illustrating the validity of our theorems. Finally, we  prove a Hardy-Knopp type inequality for Sugeno integral.
\\
\vspace{.5cm}\leftline{Keywords: Sugeno integral; P\'{o}lya-Knopp's inequality; Hardy-Knopp type inequality.}
\end{abstract}
\section{Introduction}
In 1974, M. Sugeno introduced fuzzy measures and Sugeno integral for the first time which was a important analytical method of measuring uncertain information \cite{20}.
Sugeno integral is applied in many fields such as management decision-making, medical decision-making, control engineering and so on.
Many authors such as Ralescu and Adams considered equivalent definitions of Sugeno integral  \cite{12}. Rom\'{a}n-Flores et al.  examined level-continuity of Sugeno integral and H-continuity of fuzzy measures \cite{13, 15}. For more details of Sugeno integral, we refer readers to \cite{aa, asa, bb, pap1, pf, Pap1}.



The study of fuzzy integral is first attributed to Rom\'{a}n-Flores et al. 
Many inequalities such as Markov's, Chebyshev's, Jensen's, Minkowski's, H\"{o}lder's and Hardy's inequalities 
have been studied by Flores-Franuli\v{c} and Rom\'{a}n-Flores for Sugeno integral (see \cite{5, 6} and their references). Recently, in \cite{dar2}, B. Daraby et al. studied some inequalties for Sugeno integral. 


In \cite{p-7}, Hardy announced and proved in \cite{p-8} a highly important classical Riemann integral inequality
\begin{eqnarray}
\label{p-1}
\int_0^{\infty}\left(\dfrac{1}{x}\int_0^x f(t)dt\right)^p dx\le \left(\dfrac{p}{p-1}\right)\int_0^\infty f^p(x)dx,
\end{eqnarray}
the so-called Hardy's inequality, where $p>1$ and $f \in L^p(0, \infty)$ is a non-negative function. The following related exponential Riemann integral inequality
\begin{eqnarray}
\label{p-2}
\int_0^\infty \exp\left(\dfrac{1}{x}\int_0^x \ln f(t) dt\right)dx<e\int_0^\infty f(x)dx,
\end{eqnarray}
holds for positive functions $f \in L^p(0, \infty)$. 
This inequality is known as Knopp's inequality. However, inequality \eqref{p-2} was certainly known before the mentioned Knopp's paper and Hardy himself claimed that it was G. P\'{o}lya who pointed it out to him earlier (probably by using the limit argument below). Therefore, we call the inequality \eqref{p-2} inequality P\'{o}lya-Knopp. It is important to note that inequalities \eqref{p-1} and \eqref{p-2} are closely related since \eqref{p-2} can
be obtained from \eqref{p-1} by rewriting it with the function $f$ replaced by $f^{1/p}$ and letting
$p\to\infty$. Therefore, P\'{o}lya-Knopp's inequality may be considered as a limiting relation of Hardy's inequality.


In \cite{p-10}, Kaijser et al. pointed out that both \eqref{p-1} and \eqref{p-2} are just special cases of the much more general Hardy-Knopp type inequality for positive functions $f$,
\begin{eqnarray}
\label{p-3}
\int_0^\infty \phi\left(\dfrac{1}{x}\int_0^x f(t)dt\right)\dfrac{dx}{x}\le \int_0^\infty \phi(f(x))\dfrac{dx}{x},
\end{eqnarray}
where $\phi$ is a convex function on $(0, \infty)$. This shows that both Hardy's and P\'{o}lya-Knopp's inequalities can be derived by using only convexity and gives an elegant new proof of these inequalities.

In this paper, we intend to prove P\'{o}lya-Knopp's and Hardy-Knopp's inequalities for the Sugeno integral.
This paper is organized as follows: in Section 2 some the preliminaries are  presented. In Section 3  we propose the P\'{o}lya-Knopp's inequality for Sugeno integral in two cases and we investigate Hardy-Knopp's inequality for Sugeno integral. Finally, in the last section, we presented a short conclusion.

\section{Preliminaries}
In this section, we will provide some definitions and concepts for the next sections.

Throughout this paper, we let $X$ be a non-empty set and $\Sigma$ be a $\sigma-$algebra of subsets of $X$. 
\begin{definition}(Ralescu and Adams \cite{12}).
A set function $\mu:\Sigma\to[0, +\infty]$ is called a fuzzy measure if the following properties are satisfied:
\begin{enumerate}
\item
$\mu(\emptyset)=0 $;
\item
 $A \subseteq B\Rightarrow\mu(A)\leq \mu(B)$ (monotonicity);
\item
 $A_1 \subseteq A_2 \subseteq\ldots\Rightarrow\lim \mu(A_i)=\mu\left(\bigcup \limits_{i=1}^\infty A_i\right)$ (continuity from below);
\item
  $A_1 \supseteq A_2 \supseteq \ldots$ and $\mu(A_1)<\infty\Rightarrow\lim \mu(A_i)=\mu\left(\bigcap\limits_{i=1}^\infty A_i\right)$  (continuity from above).
\end{enumerate}
When $\mu$ is a fuzzy measure, the triple $(X, \Sigma, \mu)$ is called a fuzzy measure space.
\end{definition}

If $f$ is a non-negative real-valued function on $X$, we will denote
$F_\alpha =\left\{x\in X \mid f(x)\geq\alpha\right\}=\left\{f\geq\alpha\right\}$, the
$\alpha$-level of $f$, for $\alpha>0$. The set
$F_0 = \overline{\left\{x\in X \mid f(x)>0\right\}}={\rm supp}(f)$  is the  support of $f$.

If $\mu$ is a  fuzzy measure on $X$, we define the following:
\begin{equation*}
 \mathfrak {F}^{\sigma}(X)=\left\{f : X\rightarrow[0,\infty)|\quad f~ {\rm is}\ \mu-{\rm measurable}\right\}.
\end{equation*}
\begin{definition}\label{df}
(Pap \cite{pap1}, Wang and Klir \cite{21}).
 Let $\mu$ be a fuzzy measure on
$(X, \Sigma)$. If $f \in \mathfrak{F}^\sigma(X)$ and $A\in \Sigma$, then
the Sugeno integral of $f$ on $A$ is defined by
$$ - \hspace{-1.1em} \int_A fd\mu =\bigvee_{\alpha \geq 0} \left(\alpha \wedge\mu(A\cap F_\alpha)\right),$$
where $\vee$ and $ \wedge$ denotes the operations $sup$ and $inf$ on
 $[0,\infty]$, respectively and $\mu$ is the Lebesgue measure. If $A=X$, the fuzzy integral may also be denoted by $- \hspace{-1em} \int fd\mu$.
\end{definition}



The following proposition gives the most elementary properties of the Sugeno integral.

\begin{proposition}\label{p23}
 (Pap \cite{pap1}, Wang and Klir \cite{21}).
  Let $(X, \Sigma, \mu)$ be a fuzzy measure space, $A, B \in \sum$ and $f, g \in  \mathfrak {F}^{\sigma}(X)$. We have
\begin{enumerate}
\item
 $- \hspace{-.9em} \int_A f d \mu \leq \mu(A)$;
\item
 $- \hspace{-.9em} \int_A k d \mu = k \wedge \mu(A),$ for any constant $k\in[0, \infty)$;
\item
 $- \hspace{-.9em} \int_A fd \mu < \alpha \Leftrightarrow $ there exists $\gamma< \alpha$ such that $(A \cap \{f \geq \gamma\}) < \alpha$;
\item
 $- \hspace{-.9em} \int_A fd \mu > \alpha \Leftrightarrow $ there exists $\gamma > \alpha$ such that $(A \cap \{f \geq \gamma \})> \alpha$.
\end{enumerate}
\end{proposition}

\begin{remark}
Consider the distribution function $F$ associated to $f$ on $A$, that is to say,
 $$F(\alpha)=\mu(A \cap \{f\ge\alpha\}).$$
 Then 
$$F(\alpha)=\alpha \Rightarrow - \hspace{-1.1em} \int_A f d \mu=\alpha.$$
Thus, from a numerical (or computational) point of view, the Sugeno integral can be calculated by solving the equation $F(\alpha)=\alpha$ (if the solution exists).
\end{remark}

\begin{notation}
We will use ${\rm SINT}~f(x)dx$ for the Sugeno integral on $[0, \infty)$ with respect to standard Lebesgue measure.
\end{notation}
\section{Main results}
 In this section, we prove P\'{o}lya-Knopp type inequality in two cases and Hardy-Knopp type inequality for Sugeno integral.

\begin{theorem}\label{t0-1}
(P\'{o}lya-Knopp type inequality for Sugeno integral: first case).
Let $f:[0, \infty)\to[0, \infty)$ be an increasing measurable function and
${\rm SINT}~ f(x)dx<\infty$. Then the inequality
\begin{eqnarray}
\label{t1-1}
 {\rm SINT}~\exp\left(\dfrac{1}{x}\int_0^x \ln f(t)dt\right)dx \le  {\rm SINT}~ f(x)dx,
\end{eqnarray}
holds.
\end{theorem}\begin{proof}
%
Let
$\alpha={\rm SINT}~f(x)dx$.
If 
\begin{eqnarray}
\label{t1-2}
{\rm SINT}~\exp\left(\dfrac{1}{x}\int_0^x \ln f(t)dt\right)dx>\alpha,
\end{eqnarray}
from proposition \ref{p23} (4), there exists $\gamma>\alpha$, such that
\begin{eqnarray}
\label{t1-3}
\mu\left\{\exp\left(\dfrac{1}{x}\int_0^x \ln f(t)dt\right)>\gamma\right\}>\alpha.
\end{eqnarray}
Now, if
 $$x\in \left\{\exp\left(\dfrac{1}{x}\int_0^x \ln f(t)dt\right)>\gamma\right\},$$
then we have
\begin{eqnarray}
\dfrac{1}{x}\int_0^x \ln f(t)dt>\ln \gamma.
\end{eqnarray}
By multiplying $x$ on both sides of the above equation, we get
$$\int_0^x \ln f(t)dt>x\ln\gamma=\int_0^x \ln \gamma dt.$$
From properties of classical integral, we can write
\begin{eqnarray}
\label{t1-4}
\ln f(t)>\ln\gamma.
\end{eqnarray}
Using \eqref{t1-4} and assumptions of the theorem ($f$ is increasing function), we have
$$f(x)>\gamma.$$
Therefore
$$\left\{f(x)>\gamma\right\}\supseteq \left\{\exp\left(\dfrac{1}{x}\int_0^x \ln f(t)dt\right)>\gamma\right\}.$$
From monotonicity of $\mu$, we can write
\begin{eqnarray}
\label{*}
\mu\left\{f(x)>\gamma\right\}\ge \mu\left\{\exp\left(\dfrac{1}{x}\int_0^x \ln f(t)dt\right)>\gamma\right\}.
\end{eqnarray}
Thereby, from relations \eqref{t1-3} and \eqref{*}, we obtain that
$$\mu\left\{f(x)>\gamma\right\}\ge\alpha.$$
Using the above relation and proposition \ref{p23} (4), we get
$${\rm SINT}~ f(x)dx>\alpha.$$
This is a contradiction with our initial hypothesis.
\end{proof}
In the following by an example, we illustrate the validity of Theorem \ref{t0-1}.
\begin{example}
Let $f(x)=\dfrac{x}{2}$ and $X=[0, 5]$. A straightforward calculus shows that
$$\int_0^x \ln f(t)dt=\int_0^x \ln \dfrac{t}{2} dt=t\left(\ln\left(\frac{t}{2}\right)-1\right)\bigg|_0^x=x\left(\ln\left(\frac{x}{2}\right)-1\right),$$
by multiplication $\dfrac{1}{x}$ for both sides of above relation, we get
\begin{eqnarray*}
&& \dfrac{1}{x}\int_0^x \ln f(t)dt=\ln\left(\frac{x}{2}\right)-1,\\&\Rightarrow &
\exp\left(\dfrac{1}{x}\int_0^x \ln f(t)dt\right)=\exp\left(\ln\left(\frac{x}{2}\right)-1\right)=\dfrac{x}{2}.\dfrac{1}{e}=\dfrac{x}{2e}.
\end{eqnarray*}
Now, by fuzzy integration of both sides of above equation from $0$ to $5$, we have
\begin{eqnarray*}
- \hspace{-1.1em} \int_0^5\exp\left(\dfrac{1}{x}\int_0^x \ln f(t)dt\right)dx=- \hspace{-1.1em} \int_0^5 \dfrac{x}{2e}dx.
\end{eqnarray*}
Now, we calculus $- \hspace{-.9em} \int_0^5 \dfrac{x}{2e}dx$. From Definition \ref{df}, we have
\begin{eqnarray}\label{e1-1}
- \hspace{-1.1em} \int_0^5 \dfrac{x}{2e}dx &=&
\sup_{\alpha\in[0, 5]} \left(\alpha\wedge \mu\left([0, 5]\cap F_\alpha\right)\right)\nonumber\\&=&
 \sup_{\alpha\in[0, 5]} \left(\alpha\wedge\mu\left([0, 5]\cap\left\{x:~\dfrac{x}{2e}\ge\alpha\right\}\right)\right) \nonumber\\&=&
\sup_{\alpha\in[0, 5]}  \left(\alpha\wedge\mu\left([0, 5]\cap[2e\alpha, 5]\right)\right) \nonumber\\&=&
\sup_{\alpha\in[0, 5]} \left(\alpha\wedge\mu\left([2e\alpha, 5]\right)\right) \nonumber\\&=&
\sup_{\alpha\in[0, 5]} \left(\alpha\wedge\left(5-2e\alpha\right)\right)\nonumber\\&=&
\dfrac{5}{2e+1}=0.781.
\end{eqnarray}
Finally, 
for right hand of \eqref{t1-1}, by using Definition \ref{df}, we have
\begin{eqnarray}
\label{e1-2}
- \hspace{-1.1em} \int_0^5 f(x)dx=- \hspace{-1.1em} \int_0^5 \dfrac{x}{2}dx=\dfrac{5}{3}=1.6,
\end{eqnarray}
now, from relations \eqref{e1-1} and \eqref{e1-2}, we get
$$- \hspace{-1.1em} \int_0^5  \exp\left(\dfrac{1}{x}\int_0^x \ln f(t)dt\right)dx =0.781\le  1.6=- \hspace{-1.1em} \int_0^5  f(x)dx.$$
\end{example}
\begin{theorem}\label{t0-2}
(P\'{o}lya-Knopp type inequality for Sugeno integral: second case).
Let $f:[0, \infty)\to [0, \infty)$ be a measurable function and ${\rm SINT}~ f(x)dx<\infty$. Then the inequality
\begin{eqnarray}
\label{t2-0}
{\rm SINT}~ \exp\left(\dfrac{1}{x}- \hspace{-1.1em}\int_0^x \ln f(t)dt\right)dx \le e~{\rm SINT}~ f(x)dx,
\end{eqnarray}
holds.
\end{theorem}\begin{proof}
From Proposition \ref{p23} (1), we know that
$$- \hspace{-1.1em} \int_0^x \ln f(t)dt\le x.$$
By a straightforward calculus, we obtain
\begin{eqnarray*}
 \dfrac{1}{x}- \hspace{-1.1em} \int_0^x \ln f(t)dt \le 1\Rightarrow
\exp\left(\dfrac{1}{x}- \hspace{-1.1em} \int_0^x \ln f(t)dt\right)\le \exp(1)=e,
\end{eqnarray*}
by fuzzy integration of both sides of above equation from $0$ to $\infty$ and using Proposition \ref{p23} (2), we have
\begin{eqnarray}\label{t2-1}
{\rm SINT}~ \exp\left(\dfrac{1}{x}- \hspace{-1.1em} \int_0^x \ln f(t)dt\right)dx \le {\rm SINT}~ e~ dx=e.
\end{eqnarray}
Now, the demonstration will be divided in two parts.

Case 1:
$e~{\rm SINT}~ f(x)dx> e.$
In this case, by using \eqref{t2-1}, is not difficult to see that
$${\rm SINT}~ \exp\left(\dfrac{1}{x}- \hspace{-1.1em} \int_0^x \ln f(t)dt\right)dx \le e< e~{\rm SINT}~ f(x)dx.$$

Case 2:
$e~{\rm SINT}~f(x)dx\le e.$
In this case, if we let
$q=e~{\rm SINT}~ f(x)dx$, we get
\begin{eqnarray*}
e\ge q &=&
e~{\rm SINT}~ f(x)dx\\ &\ge &
- \hspace{-1.1em} \int_0^x f(t)dt \\&=&
- \hspace{-1.1em} \int_0^x \exp\left(\ln f(t)\right)dt.
\end{eqnarray*}
By Jensen type inequality for Sugeno integral, we have
$$e\ge \exp- \hspace{-1.1em} \int_0^x\ln f(t)dt.$$
Now, we can write
\begin{eqnarray*}
1= \ln(e) \ge \ln(q) \ge - \hspace{-1.1em} \int_0^x \ln f(t)dt.
\end{eqnarray*}
By multiplying $\dfrac{1}{x}$ on both sides of the above equation, we get
\begin{eqnarray*}
&&\dfrac{1}{x}\ln(q) \ge\dfrac{1}{x} - \hspace{-1.1em} \int_0^x \ln f(t)dt \\&\Rightarrow &
\exp\left(\dfrac{1}{x}\ln(q)\right) \ge\exp\left(\dfrac{1}{x} - \hspace{-1.1em} \int_0^x \ln f(t)dt\right),
\end{eqnarray*}
by fuzzy integration of both sides of above equation from $0$ to $\infty$, we obtain
$${\rm SINT}~\exp\left(\dfrac{1}{x}\ln(q)\right)dx \ge{\rm SINT}~\exp\left(\dfrac{1}{x} - \hspace{-1.1em} \int_0^x \ln f(t)dt\right)dx.$$
Now, we calculus
${\rm SINT}~\exp\left(\dfrac{1}{x}\ln(q)\right)dx$.
For this, we have
\begin{eqnarray*}
\exp\left(\dfrac{\ln q}{x}\right)\ge\alpha\Rightarrow \dfrac{\ln q}{x}\ge \ln\alpha\Rightarrow \ln q\ge x\ln\alpha\Rightarrow x\le \dfrac{\ln q}{\ln\alpha}\Rightarrow x\in[0, \dfrac{\ln q}{\ln\alpha}].
\end{eqnarray*}
Because the integration interval is $[0, \infty)$ so
$$\dfrac{\ln q}{\ln\alpha}=\alpha\Rightarrow \ln q=\alpha\ln\alpha=\ln \alpha^{\alpha} \Rightarrow q=\alpha^{\alpha}.$$
In this equation, calculating $\alpha$ in $q$ is a little difficult.
Note that, if  we can calculus $\alpha$ in $q$, we will have a function like $\alpha=g(q)$. According to the graph of this equation, it is easy to say $q>\alpha$.
 So it's easy to write:
$${\rm SINT}~ \exp\left(\dfrac{1}{x}- \hspace{-1.1em} \int_0^x \ln f(t)dt\right)dx \le \alpha=g(q)<q=e~{\rm SINT}~ f(x)dx,$$
and the proof is now complete.
\end{proof}
In the sequel, by an example, we show the validity of the Theorem \ref{t0-2}.
\begin{example}
Let $f(x)=\exp\left(\dfrac{1}{x}\right)$. By a straightforward calculus and from Proposition \ref{p23} (2),
we have
\begin{eqnarray*}
&& - \hspace{-1.1em} \int_0^x \ln \exp\left(\dfrac{1}{t}\right)dt=1, 
\\&\Rightarrow &
\dfrac{1}{x}- \hspace{-1.1em} \int_0^x \ln \exp\left(\dfrac{1}{t}\right)dt=\dfrac{1}{x},\\ &\Rightarrow &
\exp\left(\dfrac{1}{x}- \hspace{-1.1em} \int_0^x\dfrac{1}{t} dt\right) =\exp\left(\dfrac{1}{x}\right),
\end{eqnarray*}
by fuzzy integration the above equation from $0$ to $5$, we obtain
$$- \hspace{-1.1em} \int_0^5 \exp\left(\dfrac{1}{x}- \hspace{-1.1em} \int_0^x \dfrac{1}{t} dt\right)dx= - \hspace{-1.1em} \int_0^5 \exp\left(\dfrac{1}{x}\right)dx=e,$$
and
$$- \hspace{-1.1em} \int_0^5 f(x)dx=- \hspace{-1.1em} \int_0^5 \exp\left(\dfrac{1}{x}\right)dx=e.$$
Therefore,
$$- \hspace{-1.1em} \int_0^5 \exp\left(\dfrac{1}{x}- \hspace{-1.1em} \int_0^x \ln \left(\dfrac{1}{t}\right) dt\right)dx=e \le e\times e=e- \hspace{-1.1em} \int_0^5 f(x)dx.$$
\end{example}
\begin{remark}
Generally, the Theorems \ref{t0-1} and \ref{t0-2} can be written as follows:
$${\rm SINT}~ F\left(\dfrac{1}{x}\int_0^x F^{-1}(f(t))dx\right)dx\le e~{\rm SINT}~ f(x)dx,$$
$${\rm SINT}~ F\left(\dfrac{1}{x}- \hspace{-1.1em}\int_0^x F^{-1}(f(t))dx\right)dx\le e~{\rm SINT}~ f(x)dx,$$
where
$f, F:[0, \infty)\to [0, \infty)$
(increasing measurable function in first case and measurable function in second case)
and $F^{-1}$ denotes the inverse of $F$.
Note that, $F$ must have an inverse.
\end{remark}

\begin{remark}
If we replace $\dfrac{dx}{x}$ with $dx$ in
\eqref{t2-0}, we will have
\begin{eqnarray*}
{\rm SINT}~ \exp\left(\dfrac{1}{x}- \hspace{-1.1em} \int_0^x \ln f(t)dt\right)\dfrac{dx}{x} \le e~{\rm SINT}~ f(x)\dfrac{dx}{x},
\end{eqnarray*}
that it can be rewritten in the equivalent form
\begin{eqnarray}
\label{t3-0}
{\rm SINT}~ \phi\left(\dfrac{1}{x}- \hspace{-1.1em} \int_0^x \phi^{-1} f(t)dt\right)\dfrac{dx}{x} \le e~{\rm SINT}~ f(x)\dfrac{dx}{x},
\end{eqnarray}
where $\phi^{-1}$ denotes the inverse of $\phi$,
that it's true in general case.
\end{remark}
In relation \eqref{p-3}, Kaijser et. al. presented Hardy-Knopp type inequality for classical integral, where $f$ is a positive function and $\phi$ is a convex function on $(0, \infty)$.
In the following, we prove Hardy-Knopp type inequality for Sugeno integral.
\begin{corollary}\label{t0-3}
(Hardy-Knopp type inequality for Sugeno integral)
Let $\phi$ be a positive and convex function on the range of the measurable function $f$. Then the inequality
\begin{eqnarray*}
\label{t3-1}
{\rm SINT}~ \phi\left(\dfrac{1}{x}- \hspace{-1.1em} \int_0^x f(t)dt\right)\dfrac{dx}{x}\le e~{\rm SINT}~ \phi(f(x))\dfrac{dx}{x},
\end{eqnarray*}
holds.
\end{corollary}\begin{proof}
By replacing $f(x)$ with $\phi(f(x))$ in \eqref{t3-0}, we obtain that
$${\rm SINT}~ \phi\left(\dfrac{1}{x}- \hspace{-1.1em} \int_0^x \phi^{-1} (\phi(f(t)))dt\right)\dfrac{dx}{x}\le e~{\rm SINT}~ \phi(f(x))\dfrac{dx}{x}.$$
Therefore
$${\rm SINT}~ \phi\left(\dfrac{1}{x}- \hspace{-1.1em} \int_0^x f(t)dt\right)\dfrac{dx}{x}\le e~{\rm SINT}~ \phi(f(x))\dfrac{dx}{x}.$$
\end{proof}

Observe that Theorem \ref{t0-2} can be seen as a particular case of Corollary \ref{t0-3}. Indeed, it is enough to set $\phi(x) = \exp(x)$ and $g(t) = \ln f(x)$, and hence $\phi(g(x)) = \exp(\ln f(x)) = f(x)$.

\section{Conclusion}
In this paper, we prove the P\'{o}lya-Knopp (in two cases) and Hardy-Knopp type inequalities for Sugeno integral as follows:
\begin{eqnarray*}
&&
{\rm SINT}~ \exp\left(\dfrac{1}{x}\int_0^x \ln f(t)dt\right)dx \le e~{\rm SINT}~ f(x)dx,\qquad \text{(first case)}\\&&
{\rm SINT}~ \exp\left(\dfrac{1}{x}- \hspace{-1.1em} \int_0^x \ln f(t)dt\right)dx \le e~{\rm SINT}~ f(x)dx,\qquad \text{(second case)}\\&&
{\rm SINT}~ \phi\left(\dfrac{1}{x}- \hspace{-1.1em} \int_0^x f(t)dt\right)\dfrac{dx}{x}\le e~{\rm SINT}~ \phi(f(x))\dfrac{dx}{x},
\end{eqnarray*}
where $f:[0, \infty)\to [0, \infty)$ (increasing measurable function in the first case and measurable function in the second case) and $\phi$ is convex function. In the future works,  we will discuss about these inequalities for pseudo and Choquet integrals.


\noindent

\section*{Compliance with Ethical Standards}
This article has not been funded by anyone. None of the authors received research assistance in this article. This article has not provided any studies on human participation by any of the authors.

\date{\scriptsize $^{a}$
E-mail: bdaraby@maragheh.ac.ir,}
\date{\scriptsize $^{b}$
E-mail: f.rostampour@stu.maragheh.ac.ir,}
\date{\scriptsize $^{c}$
E-mail: alirezakhodadadi@maragheh.ac.ir,\\
\date{\scriptsize $^{d}$
E-mail: rahimi@maragheh.ac.ir,
\date{\scriptsize $^{e}$
E-mail:
mesiar@math.sk
}

\end{document}